\documentclass{amsart}
\usepackage{amsthm, amsfonts, amssymb, amsmath, color}

\newtheorem{proposition}{Proposition}[section]
\newtheorem{lemma}{Lemma}[section]
\newtheorem{theorem}{Theorem}[section]
\newtheorem{corollary}{Corollary}[section]

\theoremstyle{definition}
\newtheorem{definition}{Definition}[section]
\newtheorem{remark}{Remark}[section]
\newtheorem{example}{Example}[section]

\begin{document}

\title{Basic Morse-Novikov cohomology for foliations}

\author{Liviu Ornea}\thanks{L.O. is partially supported by CNCS UEFISCDI, project
number PN-II-ID-PCE-2011-3-0118.}

 \author{Vladimir Slesar}

\date{\today }

\begin{abstract}
In this paper we find sufficient conditions for the vanishing of the  Morse-Novikov cohomology on Riemannian foliations.
We work out a Bochner technique for twisted cohomological complexes, obtaining corresponding vanishing results. Also, we generalize for our setting vanishing results from the case of closed Riemannian manifolds. Several examples are presented, along with applications in the context of l.c.s. and l.c.K. foliations.

\end{abstract}

\maketitle

\hfill



\hfill

\section{Introduction}\label{1}
We consider in what follows a smooth manifold $M$ endowed with a global closed differential one-form $\theta$. The
\emph{Morse-Novikov cohomology complex} $\left(\Omega, d_{\theta}\right)$ (where $\Omega$ is the de\thinspace Rham complex of the manifold $M$, while $d_{\theta}$ is the \emph{twisted derivative} $d_{\theta}:=d-\theta \wedge$) plays an important role
when investigating aspects related to the geometry, topology and Morse theory of the underlying manifold $M$ (see \emph{e.g}. \cite{Paj}).

Cohomological complexes of this type were also introduced and studied by Lichnerowicz in the context of Poisson geometry \cite{Lich} (in many papers this cohomology is also called \emph{Lichnerowicz cohomology}).

Classical examples of Morse-Novikov cohomology are obtained on locally conformally symplectic manifolds and locally conformally K\"ahler manifolds \cite{Orn-Verb, Vais3}. These manifolds admit local symplectic and K\"ahler structures which cannot be extended to the whole manifold. Instead, at the global level a closed one-form is obtained (called the \emph{Lee form}), and a Morse-Novikov cohomology naturally appears.

Twisted differential operators and Morse-Novikov cohomology can be canonically extended in the larger framework represented by
Riemannian foliations (\emph{i.e.} foliations with Riemannian structure that locally induces Riemannian submersions \cite{Re}).
The transverse geometry of the foliations represents the extension of the geometry of Riemannian manifolds; the classical setting is obtained in the absolute case of manifolds foliated by points \cite{Mo, To}.

On Riemannian foliations defined on a closed manifold, such twisted basic cohomological objects are mostly used
for the case when $\theta$ is related to the \emph{mean curvature form}. For example, in \cite{Do} the author studies the tenseness of Riemannian foliations (\emph{i.e.} the existence of a metric with basic (projectable) mean curvature form). In \cite{Hab-Ric}, twisted (modified) differentials are used as an alternative to the classical approach in order to investigate the tautness of the foliation (\emph{i.e.} the existence of a metric which turns all leaves into minimal submanifolds), and to perform basic harmonic analysis.

In this paper we stick with the general case of basic Morse-Novikov cohomology associated to a basic, closed one-form, and present several instances in which the groups of this cohomology are trivial.

First of all, we generalize to our setting previous vanishing results known  in the classical case when $\theta$ is parallel \cite{Leon} and, respectively, non-exact \cite{Lich-Gued}. Then we study the influence of the basic curvature on the basic Morse-Novikov cohomological groups. In addition to the case when this curvature operator in non-negative, in the second part of the paper we work out a Bochner-type technique specific for our particular framework.

Consequently, we obtain vanishing results which may hold even in the case when the basic curvature is not necessarily non-negative (in this case the basic de\thinspace Rham cohomology groups may not be trivial).

Concerning the last result, two particular cases are relevant. These are the case of classical closed Riemannian manifolds and the case of the basic de\thinspace Rham cohomology of Riemannian foliations. On the other hand, the above vanishing results have analogues in the context of locally conformally symplectic and locally conformally K\"ahler foliations.

The paper is organized as follows. In the next section we present the main features of Riemannian foliations with basic mean curvature form, which represent our framework throughout this paper. In section 3 we introduce several technical tools we use in the rest of the paper. More  precisely, the three subsections present the twisted Bott connection, the twisted basic curvature operator and a corresponding Weitzenb\"ock formula. In section 4 we present the main results of the paper along with several examples. The consequences of these results for the setting of locally conformally symplectic and locally conformally K\"ahler foliations are briefly
stated in the final section.
\section{Preliminaries}\label{2}

\subsection{Basic facts about foliations}\label{2.1}

We consider in what follows a smooth, closed Riemannian
manifold $(M,g,\mathcal{F})$ endowed with a foliation $\mathcal{F}$ such
that the metric $g$ is bundle-like \cite{Re}; the dimension of $M$ will be
denoted by $n$. We denote by $T\mathcal{F}$ the leafwise distribution
tangent to leaves. A classical vector bundle constructed on $M$ is $Q:=TM/T\mathcal{F}$ (see \emph{e.g.} \cite{To}). Note that
the definition of $Q$ does not require the metric $g$. Considering $g$, we can further obtain $Q\simeq T\mathcal{F}^{\perp}$.
For convenience, in the following we denote also the transverse distribution $T\mathcal{F}^{\perp }$ by $Q$, in accordance with \cite{Al}. Assume $\dim T\mathcal{F}=p$, $\dim Q=q$, so $p+q=n$.

As a consequence, we get the following exact sequence of vector bundles
\[
0\longrightarrow T\mathcal{F}\longrightarrow TM \longrightarrow Q \longrightarrow 0\:. \nonumber
\]

A corresponding exact sequence for the dual vector bundles also appears (see \emph{e. g.} \cite{To}).
The canonical projection operators on  the distributions $Q$ and $T\mathcal{F}$
will be denoted by $\pi_Q$ and $\pi_{T\mathcal{F}}$, respectively.

Throughout this paper we use local vector fields $\left\{ e_i,f_a\right\} $
defined on a neighborhood of an arbitrary point $x\in M$, so that they
determine an orthonormal basis at any point where they are defined,
$\left\{e_i\right\} $ spanning the distribution $Q$ and $\left\{ f_a\right\} $
spanning the distribution $T\mathcal{F}$. In what follows we use the
classical `musical' isomorphisms $\sharp $ and $\flat $ determined by the
metric structure $g$. The coframe $\left\{ e^i,f^a\right\} $ will be also
employed, with $e^i:=e_i^{\flat}$, $f^a:=f_a^{\flat }$.

A standard linear connection used in the study of the basic geometry of
our Riemannian foliated manifold is the \emph{Bott connection} (see \emph{e.g.}
\cite{To}); it is a metric and torsion-free connection traditionally defined on the smooth sections of the
quotient bundle $TM/T\mathcal{F}$. According to our above considerations, we will define
the Bott connection $\nabla$ on the sections of $Q$ by the following relations
\[
\left\{
\begin{tabular}{l}
$\nabla _uw:=\pi _Q\left( [u,w]\right) $, \\
$\nabla _vw:=\pi _Q\left( \nabla^{M} _vw\right) $,
\end{tabular}
\right.
\]
for any smooth section $u\in \Gamma \left( T\mathcal{F}\right) $ and $v$, $w\in \Gamma \left( Q\right) $.

The transverse divergence associated to the Bott connection is defined in the usual manner, as a trace operator :
$$\mathrm{div}^\nabla :=\sum_i g(\nabla_{e_i}\cdot, e_i).$$

As in the case of Riemannian submersions, we investigate the geometric
objects that can be locally projected on submanifolds transverse to the leaves.
The restriction of the classical de\thinspace Rham complex of differential forms
$\Omega \left( M\right) $ to the complex of basic (projectable) differential
forms generates the \emph{basic de\thinspace Rham complex}, defined as

\[
\Omega _b\left( \mathcal{F}\right) :=\left\{ \eta \in \Omega \left(
M\right) \mid \iota _v\eta =0,\mathcal{L}_v\eta =0\mbox{\,for any\,}v\in \Gamma \left(T\mathcal{F} \right) \right\} .
\]
Here $\mathcal{L}$ is the Lie derivative along $v$, while $\iota $ stands for interior product.
The \emph{basic de\thinspace Rham derivative} is defined also as a restriction of the
classical derivative $d$, namely $d_b:=d_{\mid \Omega _b\left( \mathcal{F}\right)}$
(see \emph{e.g.}  \cite{To}).

\begin{remark}
The basic de\thinspace Rham complex is defined independently of
the metric structure $g$, and in fact the groups of the basic de\thinspace Rham cohomology
are topological invariants \cite{To}.
\end{remark}

\begin{remark}\label{btrans}
A direct computation shows that basic forms are parallel along the leaves with respect to the connection $\nabla $ associated to a bundle-like metric. A \emph{basic vector field} is a vector field $v$ parallel along leaves with respect to $\nabla $. If the vector field is also transverse, \emph{i.e.} it is a section in the transverse distribution, $v\in \Gamma(Q)$, then $v^{\flat }\in \Omega_b(\mathcal{F}) $. In the sequel, we use  basic (projectable) vector fields $\{ e_i \}$.
\end{remark}

For any basic forms $\alpha _1$, $\alpha _2\in \Omega _b^p(\mathcal{F})$,
we extend the notation $g\left( \alpha _1,\alpha_2\right) $
for the inner product canonically induced by the metric tensor
$g $ on $\Omega _b^p(\mathcal{F})$. Taking the integral on the closed
manifold $M$, we obtain the classical $L^2$ inner product
\[
\left\langle \alpha _1,\alpha _2\right\rangle :=\int_Mg\left( \alpha_1,\alpha _2\right) d\mu _g,
\]
where $d\mu _g$ is the measure induced on $M$ by $g$.

An interesting example of differential form which is not necessarily basic
is represented by the \emph{mean curvature form} (see \emph{e.g.} \cite{Al, Bad-Esc-Ianus}). It is denoted by $\kappa $
and it is defined as
$$\kappa ^{\sharp }:=\pi _Q( \sum_a \nabla^{M}_{f_a}f_a).$$

According to \cite{Al}, on any Riemannian foliation the mean curvature form
can be decomposed in a unique way as the sum
\[
\kappa =\kappa _b+\kappa _o ,
\]
where $\kappa_b$ is the $L^2$\--orthogonal projection of $\kappa$ onto the closure of
$\Omega _b\left( \mathcal{F}\right)$, with $\kappa _o$ being its orthogonal
complement. We note that $\kappa_b$ is smooth and closed \cite{Al}.

A fundamental result in this field is due to Dom\'\i nguez \cite{Do}, and it
shows that $\kappa _o$ can always be considered to be $0$. More precisely, any
Riemannian foliation defined on a closed manifold can be turned into a
foliation with basic mean curvature form by changing the bundle-like
metric such that the transverse metric remains unchanged. As we plan to work
with geometric objects related to the transverse metric structure, we
can make the standard assumption of a basic mean curvature without
actually restricting our framework (see \emph{e.g.} \cite{Hab-Ric, To}).

And hence, from now on we shall assume that $$\boxed{\kappa = \kappa_b}$$

We canonically extend the Bott connection $\nabla$ on $\Omega_b\left(\mathcal F \right)$;
this extension is denoted by $\nabla$ as well. The adjoint operator of $d_b$ with
respect to the above hilbertian product, which is called the \emph{basic de\thinspace Rham coderivative},
may be also written as \cite{Al}
\[
\delta _b:=\sum\limits_i-\iota _{e_i}\nabla _{e_i}+\iota _{\kappa^{\sharp}} .
\]

\subsection{Morse-Novikov cohomology with basic form}\label{2.2}
On a closed Riemannian manifold $M$, using a closed differential $1\--$form
$\theta $, we can define the twisted de\thinspace Rham derivative $d_\theta :\Omega
(M)\rightarrow \Omega (M)$,
\[
d_\theta :=d-\theta \wedge ,
\]
where $\Omega (M)$ is the de\thinspace Rham complex defined on $M$.

As $\theta $ is closed,  $d_\theta ^2=0$, and the Morse-Novikov
cohomological complex $\left( \Omega (M),d_\theta \right) $ can be defined
canonically. Note that if $\theta$ is exact, $\theta=df$, $f\in \mathcal{C}^{\infty} \left( M\right)$,
then the mapping $\left[ \alpha \right] \rightarrow \left[ e^{-f}\alpha \right] $ defines an isomorphism
between the de\thinspace Rham and Morse-Novikov cohomologies.

Also, the concept of Morse-Novikov cohomology can be easily extended to the basic Morse-Novikov
cohomology on a Riemannian foliation. More
precisely, assuming that $\theta $ is a basic closed one-form, then we can
write the {\em twisted basic de\thinspace Rham derivative} using the above defined basic
de\thinspace Rham differential operator $d_b$ (see \emph{e.g.}  \cite{Ida-Pop})
\[
d_{b,\theta} :=d_b-\theta \wedge ,
\]
and the basic Morse-Novikov complex $\left( \Omega _b(\mathcal{F}),d_{b,\theta }\right) $
is constructed. The Morse-Novikov cohomology groups
$\{H_{b,\theta }^i\left( \mathcal{F}\right)\}_{0\le i\le q}$
(which can be regarded as twisted de\thinspace Rham cohomology groups) are defined in the
usual way.

The particular case $\theta =\frac 12\kappa $ is investigated in \cite{Hab-Ric},
the authors obtaining vanishing results using curvature-type operators and
an interesting interplay with the tautness properties of the foliation.

For general $\theta $, in order to study these cohomological groups, we define
basic twisted differential operators compatible with $d_{b,\theta} $.

We introduce these operators in the next section.

\section{Basic twisted differential operators on Riemannian foliations}\label{3}
\subsection{The twisted Bott connection}\label{3.1}
The main tool we use to describe and investigate the twisted cohomology
groups $H_{b,\theta }^i\left( \mathcal{F}\right) $ is a linear connection
that we modify in a convenient way.

\begin{definition}
For a closed one-form $\theta \in \Omega _b\left( \mathcal{F}\right) $ and
vector fields $v\in \Gamma \left( TM\right) $, $w\in \Gamma \left( Q\right) $
we define the {\em twisted Bott connection} $\nabla ^\theta $
\[
\nabla _v^\theta w:=\nabla _vw-\theta(v)w .
\]
\end{definition}

For the particular case when $\theta =\frac 12\kappa $ we adopt the notation
$\nabla ^{\frac 12\kappa }:=\tilde \nabla $. As this connection will play an
important role in our further considerations, we choose to denote
\begin{equation}
\tilde \nabla ^\theta :=\nabla ^{\frac 12\kappa +\theta }. \label{nabla_tilde}
\end{equation}

The connection in (\ref{nabla_tilde}) is extended canonically to $\Omega _b\left( \mathcal{F}\right)$,
and for convenience it will be denoted as $\tilde \nabla^\theta $, too.

\begin{remark}
\label{tilde_parallel}As $\kappa ^{\sharp }$, $\theta ^{\sharp }\in \Gamma
\left( Q\right) $,  if $\alpha \in \Omega _b\left(
\mathcal{F}\right) $ and $v\in \Gamma \left( T\mathcal{F}\right) $, then
\[
\tilde \nabla _v^\theta \alpha =0,
\]
in other words, {\em the basic forms are `leafwise' parallel with respect to the new
connection $\tilde \nabla ^\theta $}.
\end{remark}

For the complementary case we have the following result.

\begin{lemma}\label{nabla_basic} If $v$ is a basic vector field, then the operators $\nabla_v$ and $\tilde \nabla^{\theta}_v$ map basic forms to basic forms.
\end{lemma}
\begin{proof} We first prove the statement for $\nabla$ and basic one-forms.  Let $w$ be a basic and transverse vector field. As, locally, a Riemannian foliation can be identified with a Riemannian submersion, from \cite[Lemma 1.(3)]{ON} and the definition of $\nabla$ we see that $\nabla_v w$ is basic and transverse. By Remark \ref{btrans},  one-forms can be related to basic and transverse vector fields using the musical isomorphism induced by the bundle-like metric $g$. As the  connection $\nabla$ is metric, $\nabla_v$ maps basic one-forms to basic one-forms. We use the local coframe $\{e^i,f_a\}$ with $\{ e^i\}$ basic forms. Any basic form $\alpha$ of dimension $u$, $1\le u \le p$, can be written

\[
\alpha=\sum_{1\le i_1<..<i_u\le p} f_{i_1,..,i_u}e^{i_1}\wedge..\wedge  e^{i_u},
\]
with smooth functions $f_{i_1,..,i_u}$. We obtain

\begin{eqnarray*}
\nabla _v\alpha  &=&\sum_{1\le i_1<..<i_u\le p}v(f_{i_1,..,i_u})e^{i_1}\wedge ..\wedge e^{i_u} \\
&&+\sum_{1\le i_1<..<i_u\le p}f_{i_1,..,i_u}\nabla _ve^{i_1}\wedge ..\wedge e^{i_u}+... \\
&&+\sum_{1\le i_1<..<i_u\le p}f_{i_1,..,i_u}e^{i_1}\wedge ..\wedge \nabla _ve^{i_u}.
\end{eqnarray*}

Using \cite[Proposition 2.2]{Mo}, the functions $v(f_{i_1,..,i_u})$ are basic. Then $\nabla_v$ maps basic forms to basic forms.

As for $\tilde \nabla^{\theta}_v$, note that $\theta$, $\kappa \in \Omega_b \left(\mathcal F \right)$  and $v$ are basic, and hence $(\frac 12\kappa+\theta)(v)$ is a basic function (constant on the  leaves), so that the connection $\tilde \nabla^{\theta}$  maps basic forms to basic forms.
\end{proof}

The interesting feature of the twisted Bott connection $\tilde \nabla^\theta $ is that it can be used to build up the twisted basic de\thinspace Rham derivative.  Denote \cite{Ida-Pop}
\[
\tilde d_{b,\theta }:=d_{b,\frac 12\kappa +\theta } .
\]

An alternative way to construct this operator is the following:
\begin{equation}
\tilde d_{b,\theta } =\sum_ie^i\wedge \tilde \nabla _{e_i}^\theta . \label{d_tilde_theta}
\end{equation}
We will also denote the cohomology groups associated to $\tilde d_{b,\theta} $
by $\tilde H_{b,\theta }^i\left( \mathcal{F}\right) $, with $\tilde
H_{b,\theta }^i\left( \mathcal{F}\right) =H_{b,\frac 12\kappa +\theta}^i\left( \mathcal{F}\right) $.

For the particular case $\theta \equiv 0$, one reobtains the basic modified
operator $\tilde d_b:=\tilde d_{b,0}$, with $\tilde{d}_{b,\theta}=\tilde{d}_b-\theta\wedge$,
and the corresponding cohomology complex
$\tilde H_b^i\left(\mathcal{F}\right):=\tilde H_{b,0}^i\left( \mathcal{F}\right)$,
in accordance with \cite{Hab-Ric}.

\subsection{Computational properties of the twisted basic operators}\label{3.2}

In the following we investigate several computational properties of the above introduced  twisted
basic operators (again, for the case $\theta \equiv 0$ see also \cite{Hab-Ric, Sles}).

Consider a basic vector field $v$. Using Lemma \ref{nabla_basic}, we see that the operators
$\tilde\nabla^{\theta}$ and $\tilde d_{b,\theta}$ send basic forms to basic forms. In the following we construct the \emph{basic adjoint operators} $\tilde \nabla _v^{\theta *}$ and $\tilde \delta _{b,\theta }:=\tilde d_{b,\theta }^{*}$, \emph{i.e.} the adjoint operators with respect to the restriction of the above $L^2$ inner product to the complex of basic forms; in the remaining part of the paper each time by adjoint operators we will understand basic adjoint operators. For the twisted Bott connection we obtain the
following result.

\begin{lemma}
\label{adj_div_bar}The (formal) adjoint operator associated to the differential operator $\tilde \nabla _v^\theta $ can be
computed with the formula
\begin{equation}
\tilde \nabla _v^{\theta *}=-\tilde \nabla _v^\theta -\mathrm{div}^\nabla v-2\theta(v) .  \label{diverg}
\end{equation}
\end{lemma}

\begin{proof} Consider the operator $T:=-\tilde \nabla _v^\theta -\mathrm{div}^\nabla v-2\theta(v)$. For all smooth basic forms $\alpha_1$ and $\alpha_2$ we evaluate the expression
\[
S:=\ \left\langle \tilde \nabla _v^\theta \alpha _1,\alpha _2\right\rangle
-\left\langle \alpha _1,T \alpha_2\right\rangle .
\]

We obtain
\begin{equation}\label{adjoint_diverg_1}
\begin{split}
\left\langle \tilde \nabla _v^\theta \alpha _1,\alpha _2\right\rangle
&=\left\langle \nabla _v\alpha _1,\alpha _2\right\rangle -\frac
12\left\langle \kappa(v) \alpha _1,\alpha_2\right\rangle  \\
&-\left\langle \theta(v) \alpha _1,\alpha_2\right\rangle
\end{split}
\end{equation}
and
\begin{equation}\label{adjoint_diverg_2}
\begin{split}
\left\langle \alpha _1,T \alpha_2\right\rangle &=-\left\langle \alpha _1,\nabla _v\alpha _2\right\rangle +\frac12\left\langle \alpha _1,\kappa(v) \alpha_2\right\rangle   \\
& -\left\langle \alpha _1, \theta(v) \alpha_2\right\rangle
-\left\langle \alpha _1,(\mathrm{div}^\nabla v)\alpha _2\right\rangle .
\end{split}
\end{equation}
From (\ref{adjoint_diverg_1}) and (\ref{adjoint_diverg_2}), similarly to the classical case, we get
\begin{equation*}
\begin{split}
\ S &=\int_Mv \left(g( \alpha _1,\alpha _2)\right) d\mu _g
+\int_M\mathrm{div}^\nabla v\,g( \alpha _1,\alpha _2) d\mu _g \\
& +\int_Mg( \nabla^{M}_{f_a}v,f_a) \,g( \alpha_1,\alpha _2) d\mu _g \\
 &=\int_M\left( v\left( g( \alpha _1,\alpha _2) \right) +
\mathrm{div}v\,g( \alpha _1,\alpha _2) \right) d\mu _g \\
&=\int_M\mathrm{div}\left( g( \alpha _1,\alpha _2) v\right) =0\quad\text{by Green theorem, \cite
{Po}}.
\end{split}
\end{equation*}

Then, $T$ is in fact the formal adjoint operator of $\tilde{\nabla}_v^{\theta}$, so
\[
\tilde \nabla _v^{\theta *}=-\tilde \nabla _v^\theta -\mathrm{div}^\nabla v-2\theta(v) .
\]

\end{proof}

We now compute  the adjoint operator $\tilde{\delta}_{b,\theta}=\tilde{d}^{*}_{b,\theta}$. Consider first the operator
\begin{equation*}
\begin{split}
\tilde \delta _{b} &:=\sum_i-\iota _{e_i}\nabla _{e_i}+\frac
12\iota _{\kappa ^{\sharp }} \\
&=\delta _b-\frac 12\iota _{\kappa ^{\sharp }}
\end{split}
\end{equation*}
which is known to be the adjoint of $\tilde d_{b}$, \cite{Hab-Ric}. Then, the operator
\begin{equation}
\begin{split}
\tilde \delta _{b,\theta} &:=\tilde \delta _{b}-\iota _{\theta ^{\sharp }}  \label{delta_tilde_theta}\\
&=\sum_i-\iota _{e_i}\tilde \nabla_{e_i}^\theta -2\iota _{\theta ^{\sharp }}
\end{split}
\end{equation}
is the adjoint of $\tilde{d}_{b,\theta }$.

Let $v$ and $w$ be transverse basic vector fields and let $\alpha$ be a basic
differential form. We define the Clifford product
\[
v\cdot \alpha :=v^{\flat }\wedge \alpha -\iota _v\alpha .
\]
We remark that $v\cdot v\cdot \alpha=-\left\| v \right\|^2_g \alpha,$ where
$\left\| v \right\|_g:=\sqrt{g\left(v,v \right)}$.

We define the corresponding Dirac-type operator $\tilde D_{b,\theta }$ using (\ref{d_tilde_theta}) and (\ref{delta_tilde_theta})
\begin{equation*}
\begin{split}
\tilde D_{b,\theta } &:=\tilde d_{b,\theta }+\tilde \delta _{b,\theta } \\
&=\sum_i( e^i\wedge \tilde \nabla _{e_i}^\theta -\iota_{e_i}\tilde \nabla
_{e_i}^\theta ) -2\iota _{\theta ^{\sharp }} \\
\ &=\sum_ie_i\cdot \tilde \nabla _{e_i}^\theta -2\iota _{\theta ^{\sharp }}.
\end{split}
\end{equation*}

As in the classical case, a Laplace-type operator related to $\tilde
d_{b,\theta }$ can be defined
\[
\tilde \Delta _{b,\theta } :=\tilde d_{b,\theta }\tilde \delta _{b,\theta }+\tilde \delta _{b,\theta}\tilde d_{b,\theta }. \nonumber
\]

\begin{remark}
$\tilde \Delta _{b,\theta }$ is a transverse
elliptic operator defined on the Riemannian foliation, with the same symbol
as $\Delta _b$.

For $\theta =0$, we obtain the twisted operator $\tilde \Delta _b$ employed
in \cite{Hab-Ric}; furthermore, for $\theta =-\frac 12\kappa $ we actually
obtain the basic Laplace operator $\Delta _b$ (see \emph{e.g.}  \cite{Ric-Park, To}).
\end{remark}

$( \Omega _b ( \mathcal{F} ) , \tilde d_{b,\theta }) $ is a transverse elliptic complex and hence, similarly to \cite[Proposition 2.3]{Hab-Ric}
(for the classical case when the manifold is foliated by points see \emph{e.g.}  \cite{Gil, Vais2}), the following Hodge-type
decomposition holds:

\begin{theorem}{\rm (\cite{Ida-Pop})}
The basic cohomology $\Omega _b(\mathcal{F})$ can be written as a direct sum
\[
\Omega _b(\mathcal{F})=\mathrm{Im}( \tilde d_{b,\theta }) \oplus
\mathrm{Im}( \tilde \delta _{b,\theta }) \oplus \mathrm{Ker}( \tilde \Delta _{b,\theta }) .
\]
\end{theorem}

If we denote $\mathcal{H}^p ( \tilde \Delta _{b,\theta })
:=\mathrm{Ker}\tilde \Delta _{b,\theta \mid \Omega _b(\mathcal{F})}$, with
$0\le p\le n $, then
\[
\mathcal{H}^p( \Delta _\theta ) \simeq \tilde H_{b,\theta}^p ( \mathcal{F})
=H_{b,\frac 12\kappa +\theta }^p( \mathcal{F}) .
\]

Now, concerning the twisted Bott connection and the Clifford product, we have
\begin{lemma}
\label{Leibniz}The following Leibniz rule holds:
\[
\tilde \nabla _v^\theta ( w\cdot \alpha ) =\nabla _vw\cdot \alpha
+w\cdot \tilde \nabla _v^\theta \alpha .
\]
\end{lemma}

\begin{proof} We have
\begin{equation*}
\begin{split}
\tilde \nabla _v^\theta ( w\cdot \alpha ) &=\nabla _v(
w\cdot \alpha) -(\frac 12\kappa+\theta)(v) w\cdot \alpha \\
 &=\nabla _vw\cdot \alpha +w \cdot ( \nabla _v\alpha
 -(\frac 12\kappa+\theta)(v) \alpha ),
\end{split}
\end{equation*}
and the result follows from the very definition of $\tilde \nabla _v^\theta $.
\end{proof}

\begin{lemma}
\label{interior_nabla}The following relation holds:
\[
\iota _w\tilde \nabla _v^\theta =\tilde \nabla _v^\theta \iota _w-\iota_{\nabla _vw}.
\]
\end{lemma}

\begin{proof} We can write
\begin{equation*}
\begin{split}
\iota _w\tilde \nabla _v^\theta &=\iota _w ( \nabla _v-(\frac 12\kappa+\theta)(v)) \\
&=\nabla _v\iota _w-\iota _{\nabla _vw}-(\frac 12\kappa+\theta)(v) \iota _w \\
&=\tilde \nabla _v^\theta \iota _w-\iota _{\nabla _vw}.
\end{split}
\end{equation*}
\end{proof}

The following two equations relating standard operators on Riemannian
foliations are the natural extension of classical results from calculus on
differentiable manifolds. The proofs are similar to the classical case.

\begin{lemma}\label{inter_prod_Clifford}
For any basic vector fields $v$, $w$ and basic
form $\alpha \in \Omega _b(\mathcal{F})$, we have
\begin{equation}\label{prop_calcul_class}
\begin{split}
\iota _w( v\cdot \alpha )& =\iota _wv^{\flat }\,\alpha -v\cdot
\iota _w\alpha ,   \\
\mathcal{L}_v\alpha -\nabla _v\alpha &=\sum_ie^i\wedge \iota _{\nabla_{e_i}v}\alpha .
\end{split}
\end{equation}
\end{lemma}

\subsection{The twisted basic curvature operator}\label{3.3}
In this subsection we show that the curvature operator associated to the twisted Bott connection
coincides in fact with the basic curvature operator and depends only on the transverse metric.

Let $\gamma $ be a closed basic one-form  and denote by
$R_{v,w}^\gamma $ the basic curvature operator written using the
connection $\nabla ^\gamma $ and the transverse basic vector fields $v$, $w$. We have the following useful relation:

\begin{lemma}
\label{curbura}The curvature operator $R_{v,w}^\gamma $ does not depend on $\gamma $, namely
\[
R_{v,w}^\gamma =R_{v,w}.
\]
\end{lemma}

\begin{proof} Starting with the definition of the curvature operator, we obtain
for $R_{v,w}^\gamma $
\begin{equation}
R_{v,w}^\gamma =\nabla _v^\gamma \nabla _w^\gamma -\nabla _w^\gamma \nabla_v^\gamma -\nabla _{[v,w]}^\gamma .  \label{gamma_curvature}
\end{equation}
Furthermore, we compute:
\begin{equation*}
\begin{split}
\nabla _v^\gamma \nabla _w^\gamma &=( \nabla _v-\gamma(v) ) ( \nabla _w-\gamma(w)) \\
&=\nabla _v\nabla _w-\gamma(v) \nabla _w-v(\gamma(w)) \\
&-\gamma (w)\nabla(v)+\gamma(v)\gamma(w).
\end{split}
\end{equation*}
Similarly, we have:
\begin{equation*}
\begin{split}
\nabla _w^\gamma \nabla _v^\gamma &=( \nabla _w-\gamma(w) ) ( \nabla _v-\gamma(v)) \\
&=\nabla _w\nabla _v-\gamma(w) \nabla _v-w(\gamma(v)) \\
&-\gamma (v)\nabla(w)+\gamma(w)\gamma(v).
\end{split}
\end{equation*}
and
\[
\nabla _{[v,w]}^\gamma =\nabla _{[v,w]}-\gamma([v,w]) .
\]
Now, as $d\gamma=0 $, we have:
\begin{equation}
v(\gamma(w)-w(\gamma(v))=\gamma([v,w]) .  \label{closed}
\end{equation}

The conclusion follows.
\end{proof}

\begin{remark}\label{curvature}
As $\tilde R^\theta =R^{\frac 12\kappa +\theta }$, we also have $\tilde
R^\theta =R$.
\end{remark}

\subsection{A twisted basic Weitzenb\"ock formula}\label{3.4}
We present now the Weitzenb\"ock-type formula for the Laplace operator
$\tilde \Delta _{b,\theta }$, canonically constructed by using the
derivative $\tilde d_{b,\theta }$ on $\Omega _b\left( \mathcal{F}\right) $.

Note first that
\begin{equation}\label{trei termeni}
\begin{split}
\tilde \Delta _{b,\theta }&=\tilde{D}^2_{b,\theta}\\
&=\sum_{i,j}( e_i\cdot \tilde \nabla _{e_i}^\theta ) (e_j\cdot \tilde \nabla _{e_j}^\theta )
-2\sum_i\iota _{\theta ^{\sharp}}( e_i\cdot \tilde \nabla _{e_i}^\theta )   \\
&-2\sum_i( e_i\cdot \tilde \nabla _{e_i}^\theta ) \iota_{\theta ^{\sharp }}
+4\iota _{\theta ^{\sharp }}\iota _{\theta ^{\sharp }}.
\end{split}
\end{equation}
As the last term vanishes, we compute below the other three terms.

For the first one, using Lemma \ref{Leibniz}, we get
\begin{equation*}
\begin{split}
 \sum_{i,j}( e_i\cdot \tilde \nabla _{e_i}^\theta )(e_j\cdot \tilde \nabla _{e_j}^\theta )
& =\sum_{i,j}e_i\cdot \nabla_{e_i}e_j\cdot \tilde \nabla _{e_j}^\theta \\
&+\sum_{i,j}e_i\cdot e_j\cdot \tilde \nabla _{e_i}^\theta \cdot
\tilde \nabla _{e_j}^\theta .
\end{split}
\end{equation*}
Let $\Gamma _{ij}^k$ be the Christoffel coefficients given by
$\nabla _{e_i}e_j=\sum_k\Gamma _{ij}^ke_k$, for the local orthonormal
frame $\{e_i\}_{1\le i\le q}$. As $e_i$ are basic,  $\Gamma_{ij}^k$ are basic functions. Then:

\begin{equation*}
\begin{split}
\sum_{i,j}e_i\cdot \nabla _{e_i}e_j\cdot \tilde \nabla _{e_j}^\theta
&=\sum_{i,j,k}e_i\cdot -\Gamma _{ik}^je_k\cdot \tilde \nabla _{e_j}^\theta \\
&=-\sum_{i,k}e_i\cdot e_k\cdot \tilde \nabla _{\nabla _{e_i}e_k}^\theta ,
\end{split}
\end{equation*}
and we obtain
\begin{equation}\label{weitz_curvature_term}
\begin{split}
\sum_{i,j}( e_i\cdot \tilde \nabla _{e_i}^\theta )
&(e_j\cdot \tilde \nabla _{e_j}^\theta ) =\sum_i \tilde \nabla _{\nabla _{e_i}e_i}^\theta
-\sum_i\tilde \nabla _{e_i}^\theta \tilde \nabla_{e_i}^\theta   \\
&+\sum_{i<j}e_i\cdot e_j\cdot ( \tilde \nabla _{e_i}^\theta
\tilde \nabla _{e_j}^\theta -\tilde \nabla _{e_j}^\theta \tilde \nabla_{e_i}^\theta
-\tilde \nabla _{_{\nabla _{e_i}e_j}-_{\nabla_{e_j}e_i}}^\theta ) .
\end{split}
\end{equation}
As
\[
\sum_i\nabla_{e_i}e_i=-\sum_i g( \nabla _{e_i}e_j,e_i)e_j=-\sum_i \mathrm{div}^\nabla e_i \,e_i,
\]
the `rough' Laplace operator $(\tilde \nabla^\theta)^2$ of the basic twisted connection satisfies the equation:
\begin{equation*}
\begin{split}
(\tilde \nabla^\theta)^2 &=\sum_i\tilde \nabla _{\nabla _{e_i}e_i}^\theta
-\sum_i\tilde \nabla_{e_i}^\theta \tilde \nabla _{e_i}^\theta \\
&=\sum_i-\mathrm{div}^\nabla \,e_i\tilde \nabla _{e_i}^\theta -\sum_i\tilde
\nabla _{e_i}^\theta \tilde \nabla _{e_i}^\theta \\
&=\sum_i\tilde \nabla _{e_i}^{\theta *}\tilde \nabla _{e_i}^\theta
+2\sum_i \theta(e_i) \tilde \nabla _{e_i}^\theta \quad\text{by Lemma \ref{adj_div_bar}}.
\end{split}
\end{equation*}
Observe that, according to the definition of the twisted Bott connection, in (\ref
{weitz_curvature_term}) we do not have yet a basic curvature operator
associated to $\tilde \nabla ^\theta $. However, using Remark \ref{tilde_parallel}, we obtain
\[
\tilde \nabla _{\pi _{T\mathcal{F}}([e_i,e_j])}^\theta \alpha =0 ,
\]
for any $\alpha \in \Omega _b(\mathcal{F})$.
Then, using also Remark \ref{curvature}, we denote
\begin{equation*}
\begin{split}
\mathcal{R} &:=\sum_{i<j}e_i\cdot e_j\cdot( \tilde \nabla_{e_i}^\theta \tilde \nabla _{e_j}^\theta
-\tilde \nabla _{e_j}^\theta
\tilde \nabla _{e_i}^\theta -\tilde \nabla _{_{\nabla _{e_i}e_j}-_{\nabla_{e_j}e_i}}^\theta ) \\
&=\sum_{i<j}e_i\cdot e_j\cdot \tilde R_{e_{i,}e_j}^\theta \\
&=\sum_{i<j}e_i\cdot e_j\cdot R_{e_{i,}e_j}.
\end{split}
\end{equation*}
With this, we finally obtain
\begin{equation}
\sum_{i,j}( e_i\cdot \tilde \nabla _{e_i}^\theta ) (e_j\cdot \tilde \nabla _{e_j}^\theta )
=\sum_i\tilde \nabla_{e_i}^{\theta *}\tilde \nabla _{e_i}^\theta
+2\sum_i \theta(e_i) \tilde \nabla _{e_i}^\theta +\mathcal{R} . \label{termen 1}
\end{equation}
As for the second term in (\ref{trei termeni}), using Lemma \ref
{inter_prod_Clifford}, we get
\begin{equation}
\begin{split}\label{termen 2}
\sum_i\iota _{\theta ^{\sharp }}( e_i\cdot \tilde \nabla _{e_i}^\theta)
&=\sum_i\iota _{\theta ^{\sharp }}e^i\,\tilde \nabla _{e_i}^\theta
-\sum_ie_i\cdot \iota _{\theta ^{\sharp }}\tilde \nabla _{e_i}^\theta\\
 &=\sum_i \theta(e_i) \tilde \nabla _{e_i}^\theta
-\sum_ie_i\cdot \iota _{\theta ^{\sharp }}\tilde \nabla _{e_i}^\theta .
\end{split}
\end{equation}
For the third term we use Lemma \ref{interior_nabla}:
\begin{equation}
\sum_i( e_i\cdot \tilde \nabla _{e_i}^\theta ) \iota _{\theta^{\sharp }}
=\sum_ie_i\cdot \iota _{\theta ^{\sharp }}\tilde \nabla_{e_i}^\theta
+\sum_ie_i\cdot \iota _{\nabla _{e_i}\theta ^{\sharp }}.
\label{termen 3}
\end{equation}
Plugging equations \eqref{termen 1}-\eqref{termen 3} in the relation
(\ref{trei termeni}), we end up with the corresponding version of Weitzenb\"ock
formula for the Laplace-type operator $\tilde \Delta _{b,\theta }$:
\begin{equation}\label{basic_Weitzenbock formula}
\begin{split}
\tilde \Delta _{b,\theta } &=\sum_i\tilde \nabla _{e_i}^{\theta *}\tilde
\nabla _{e_i}^\theta +2\sum_i \theta(e_i) \tilde
\nabla _{e_i}^\theta +\mathcal{R}
-2\sum_i \theta(e_i)  \tilde \nabla _{e_i}^\theta   \\
&+2\sum_ie_i\cdot \iota _{\theta ^{\sharp }}\tilde \nabla _{e_i}^\theta
-2\sum_ie_i\cdot \iota _{\theta ^{\sharp }}\tilde \nabla _{e_i}^\theta
-2\sum_ie_i\cdot \iota _{\nabla _{e_i}\theta ^{\sharp }} \\
&=\sum_i\tilde \nabla _{e_i}^{\theta *}\tilde \nabla_{e_i}^\theta
-2\sum_ie_i\cdot \iota _{\nabla _{e_i}\theta ^{\sharp}}+\mathcal{R}.
\end{split}
\end{equation}

\begin{remark}
\label{remark_weitz}In general, it is difficult to obtain a convenient basic
Weitzenb\"ock formula. First of all, if $\theta =0$, then we obtain the
classical form of such formula for $\tilde \Delta _b$ (for the particular
case when the mean curvature $\kappa $ is also harmonic, see also \cite
{Hab-Ric}). Secondly, as we noticed before, for the particular case when
$\theta =-1/2\cdot\kappa $, we obtain the classical basic Laplace operator, \cite{Ric-Park}.
Assuming that $\kappa $ is not only basic, but also parallel with respect to
the Bott connection (with the corresponding topological consequences), we
obtain again a standard form for \eqref{basic_Weitzenbock formula}.
\end{remark}

In the next section we work out the Bochner technique for the basic
Morse-Novikov cohomology, deriving corresponding conditions that imply the
vanishing of the cohomology groups.

\section{Vanishing results for the basic Morse-Novikov cohomology}\label{4}

In this section we investigate several situations when the groups of
basic Morse-Novikov cohomology vanish.

We start by extending to our context a result from \cite{Leon}, where
the authors prove that if the closed form $\theta $ is also parallel, then
the Morse-Novikov cohomology becomes trivial.

We shall need the following commutation formulae, written using the operator $\tilde\nabla$. Recall that $\tilde\nabla=\nabla^{\frac 12 \kappa}=\tilde\nabla^0$.

\begin{lemma}  If the closed one-form $\theta$ is parallel with respect to the Bott connection,
then the following equations are satisfied:
$$[\tilde\nabla_{\theta^\sharp},\tilde d_{b,\theta}]=0,\qquad [\tilde\nabla_{\theta^\sharp},\tilde \delta_{b,\theta}]=0.$$
\end{lemma}
\begin{proof}
We have

\begin{equation}\label{Lie}
\begin{split}
\tilde \nabla _{\theta ^{\sharp }} &=\sum_{i}e^i\wedge \iota _{\theta^{\sharp }}\tilde \nabla _{e_i}
+\sum_{i} \theta(e_i) \tilde \nabla _{e_i}
-\sum_{i}e^i\wedge \iota _{\theta^{\sharp }}\tilde \nabla _{e_i}   \\
&=\sum_{i}e^i\wedge \tilde \nabla _{e_i}\iota _{\theta ^{\sharp}}
+\iota _{\theta ^{\sharp }}\sum_{i}e^i\wedge \tilde \nabla _{e_i}\\
&=\tilde d_b\iota _{\theta ^{\sharp }}+\iota _{\theta ^{\sharp }}\tilde d_b,
\end{split}
\end{equation}
where we use the fact that $\theta $ is parallel.

We achieve the result in two steps. Firstly, using (\ref{Lie}) we show that
$\tilde d_{b}$, $\tilde \delta_{b}$ commute with $\tilde \nabla _{\theta^{\sharp}} $.

\begin{equation}
\tilde \nabla _{\theta ^{\sharp }}\tilde d_b =\tilde d_b\iota _{\theta^{\sharp }}\tilde d_b
=\tilde d_b( \tilde \nabla _{\theta ^{\sharp }}-\tilde d_b\iota_{\theta ^{\sharp }})
=\tilde d_b\tilde \nabla _{\theta ^{\sharp }}.
\end{equation}
From Lemma \ref{adj_div_bar}, as $\theta $ is parallel with respect to the Bott connection, we derive that
$\tilde \nabla _{\theta ^{\sharp }}^{*} =-\tilde \nabla _{\theta ^{\sharp }}$. Taking adjoint operators, we  obtain:
\[
\tilde \nabla _{\theta ^{\sharp }}\tilde \delta _b=\tilde \delta _b\tilde
\nabla _{\theta ^{\sharp }}.
\]
Using the fact that $\theta$ is parallel, the following Leibniz rule is easy to prove for the connection $\tilde\nabla$ and any $v \in \Gamma(Q)$.
\[
\tilde \nabla_v \theta \wedge=(\nabla_v \theta) \wedge+ \theta \wedge \tilde \nabla_v=\theta \wedge \tilde \nabla_v .
\]

Then, for the twisted operators, this gives:
\begin{equation}\label{commut_d_theta}
\begin{split}
\tilde \nabla _{\theta ^{\sharp }}\tilde d_{b,\theta }
&=\tilde \nabla_{\theta ^{\sharp }}\tilde d_b-\tilde \nabla _{\theta ^{\sharp }}\theta
\wedge   \\
&=( \tilde d_b-\theta \wedge ) \tilde \nabla _{\theta ^{\sharp}} \\
&=\tilde d_{b,\theta }\tilde \nabla _{\theta ^{\sharp }}.
\end{split}
\end{equation}
Again taking adjoint operators, we obtain
\begin{equation}
\tilde \nabla _{\theta ^{\sharp }}\tilde \delta _{b,\theta }
=\tilde \delta_{b,\theta ^{\sharp }}\tilde \nabla _{\theta^\sharp} .
\label{commut_delta_theta}
\end{equation}
\end{proof}

The relation \eqref{Lie} allows us to establish the link between the operators
$\tilde d_{b,\theta }$, $\iota_{\theta^{\sharp}}$ and the twisted
connection $\tilde \nabla _{\theta ^{\sharp }}$, namely
\begin{equation}\label{Lie_tilde}
\begin{split}
\tilde \nabla _{\theta ^{\sharp }}-\|\theta\|^2\mathrm{Id} &=( \tilde d_b-\theta \wedge) \iota _{\theta ^{\sharp }}
+\iota _{\theta ^{\sharp }}( \tilde d_b-\theta \wedge )   \\
&=\tilde d_{b,\theta }\iota _{\theta ^{\sharp }}+\iota _{\theta ^{\sharp}}\tilde d_{b,\theta }.
\end{split}
\end{equation}
Now we can prove:

\begin{theorem}\label{Thm_2}
Let $\left( M,\mathcal{F},g\right) $ be a Riemannian
foliation with closed manifold $M$ and basic mean curvature $\kappa $. If the basic, nontrivial
$one$\--form $\theta$ is parallel with respect to the Bott connection $\nabla$, then
\[
\tilde H_{b,\theta }^i( \mathcal{F}) =H_{b,\frac 12\kappa +\theta}^i( \mathcal{F}) =0
\]
for $0\le i\le q$, where $H_{b,\frac 12\kappa +\theta }^i$ are the basic
Morse-Novikov cohomology groups.
\end{theorem}
\begin{proof}
Assume that the basic, nontrivial one-form $\theta $
is parallel with respect to the Bott connection $\nabla $. Similar to \cite{Leon},
without restricting the generality we can also assume that $\|\theta\|=1$, otherwise
considering the conformal transformation of the metric $g':=\|\theta\|^2g$ we
obtain the desired condition.
Let $\alpha \in \mathcal{H}^p( \tilde \Delta _{b,\theta}) $,
where, as above, $\mathcal{H}^p( \tilde \Delta _{b,\theta})
=\mathrm{Ker}\tilde \Delta _{b,\theta \mid \Omega _b(\mathcal{F})}$.
Then $\tilde d_{b,\theta }\alpha =0$, $\tilde \delta_{b,\theta }=0$,
$\alpha $ being a harmonic form associated to $\tilde \Delta_{b,\theta }$.
As $\theta $ is parallel with respect to the Bott connection, using again that
$\tilde \nabla _{\theta ^{\sharp }}^{*} =-\tilde \nabla _{\theta ^{\sharp }}$, we get
\[
\left\langle \tilde \nabla _{\theta ^{\sharp }}\alpha ,\alpha \right\rangle
=\left\langle \alpha ,-\tilde \nabla _{\theta ^{\sharp }}\alpha
\right\rangle ,
\]
so
\begin{equation}
\left\langle \tilde \nabla _{\theta ^{\sharp }}\alpha ,\alpha \right\rangle =0.  \label{anulare}
\end{equation}
Now, equation (\ref{Lie_tilde}) implies:
\[
\tilde \nabla _{\theta ^{\sharp }}\alpha -\alpha =\tilde d_{b,\theta }\left(
\iota _{\theta ^{\sharp }}\alpha \right) ,
\]
and $\left[ \tilde \nabla _{\theta ^{\sharp }}\alpha \right] \equiv \left[
\alpha \right] $ \emph{i.e.} $\tilde \nabla _{\theta ^{\sharp }}\alpha $ and $\alpha $
lie in the same cohomology class of $H_{b,\frac 12\kappa +\theta}^p $.
Using the commutation relations (\ref{commut_d_theta}) and
(\ref{commut_delta_theta}), we prove that $\tilde \nabla_{\theta ^{\sharp}}\alpha $
is also a harmonic form. Indeed,
\begin{eqnarray*}
\tilde d_{b,\theta }\tilde \nabla _{\theta ^{\sharp }}\alpha &=&\tilde
\nabla _{\theta ^{\sharp }}\tilde d_{b,\theta }\alpha =0, \\
\tilde \delta _{b,\theta }\tilde \nabla _{\theta ^{\sharp }}\alpha &=&\tilde
\nabla _{\theta ^{\sharp }}\tilde \delta _{b,\theta }\alpha =0,
\end{eqnarray*}
and as a consequence we must have $\tilde \nabla _{\theta ^{\sharp }}\alpha
=\alpha $. From (\ref{anulare}) it follows that $\alpha =0$, and hence
$\mathcal{H}^p( \tilde \Delta _{b,\theta }) \equiv 0$, which proves  the result.
\end{proof}

\begin{example}\label{ex1} We construct a Riemannian foliation endowed with a parallel
basic one-form and apply the above result.

 Consider a Hopf manifold constructed in the
following manner (see \emph{e.g.}  \cite{Drag-Orn, Vais1}).  On the  complex
manifold $\mathbb{C}^n\setminus\{0\}$ one considers the metric $g:=z\mapsto 1/\left|
z\right| ^2\cdot g_0$, where $g_0$ is the
canonical Euclidean metric, with $z:=\left( z^1,\ldots,z^n\right) $.  Let $\Delta $ be the
cyclic group  generated by the transformation $z\mapsto e^2z$.
The quotient $\mathbb{C}H:=\left( \mathbb{C}^n\setminus\{0\}\right) /\Delta $,
is a \emph{complex Hopf manifold}. Moreover, the above metric is
invariant with respect to the transformation, and a quotient metric, still denoted by $g$, is
induced on $\mathbb{C}H$. If $S^{1}(1/\pi )$ is the circle of radius $1/\pi$,
then the mapping $f:\mathbb{C}^n\setminus\{0\}\rightarrow S^{1}(1/\pi )\times S^{2n-1}$ defined as
\[
f(z):=\left( \frac 1\pi e^{i\pi \ln \left| z\right| },\frac z{\left|
z\right| }\right)
\]
is invariant with respect to the above group action, and induces an isometry
between $\mathbb{C}H$ and $S^{1}(1/\pi )\times S^{2n-1}$ \cite{Drag-Orn}. If $J$
is the complex structure, then the K\"ahler form $\omega\left( \cdot
,\cdot \right) :=g\left( \cdot ,J\cdot \right) $ has the expression
\[
\omega =-i\frac 1{2\left| z\right| ^2}\sum_jdz^j\wedge d\bar z^j.
\]
One easily sees that $\omega$ satisfies the equation (see Section \ref{l.c.s._l.c.K.})
\[
d\omega =\theta \wedge \omega ,\quad \text{with}\quad \theta=-\frac 1{\left| z\right| ^2}\sum_j\left(z^jd\bar z^j+\bar z^jdz^j\right).
\]

This $\theta$ is called {\em Lee form} and, in this case, it is parallel with respect to the Levi-Civita connection of $g$. Its metric dual, called the {\em Lee field}:
\[
B:=\theta ^{\sharp }=-\frac{1}{2}\sum_j \left( z^j \frac{\partial}{\partial z^j}
+\bar{z}^j \frac{\partial}{\partial \bar{z}^j} \right)
\]
is also parallel, and consequently it is
Killing, so we have (see \emph{e.g.}  \cite{Drag-Orn})
\[
\mathcal{L}_Bg=0,\mbox{\,\,}\mathcal{L}_B\theta =0.
\]
If $\varphi $ is the flow generated by $B$ and if $T\in \mathbb{R}$ is fixed,
then $\varphi _T$ is an isometry of  $\mathbb{C}H$ which leaves $g$ and $\theta $
invariant. Then, the direct product space $\hat M:=\mathbb{C}H\times \mathbb{R}$
is endowed with a direct product Riemannian structure (on $\mathbb{R}$
we just consider the standard metric). Moreover, $\hat M$ is foliated by
the real lines, and $\theta $ is in fact a basic one-form which is parallel
with respect to the Bott connection.  We can now \emph{suspend}
(see \emph{e.g.}  \cite{Mo}) the action of  $\varphi _T$ on $\mathbb{C}H$ by considering the
equivalence relation $\left( y,x\right) \sim \left( \varphi _T\left(
y\right) ,x+1\right) $ on $\hat{M}$ and taking the quotient
\[
M:=\hat M/\sim .
\]
We end up with a Riemannian foliation $\left( M,\mathcal{F}\right) $ on
which the basic one-form $\theta $ remains parallel. Applying Theorem \ref{Thm_2}, the
basic Morse-Novikov cohomology groups  $\{\tilde H_{b,\theta }^i\left(
\mathcal{F}\right)\} _{\,0\le i\le q}$ are trivial for the above suspension of the
Hopf manifold.
\end{example}

In what follows we investigate the weaker case when $\theta $ is closed, but
non-exact and non-parallel. We adapt the arguments from \cite{Lich-Gued} (see also \cite{Drag-Orn})
to our framework. More precisely, we prove the following statement:

\begin{theorem}\label{Thm_3}
Let $\left( M,\mathcal{F},g\right) $ be a transversally oriented Riemannian foliation
with the underlying manifold $M$ closed and connected, and suppose the mean curvature $\kappa$ basic. If the basic
one-form $\theta $ is closed but not exact, then the  top
dimension basic Morse-Novikov cohomology group $\tilde H_{b,\theta }^q\left( \mathcal{F}\right)$ vanishes,
\[
\tilde H_{b,\theta }^q( \mathcal{F}) = H_{b,\frac 12\kappa +\theta}^q( \mathcal{F}) =0.
\]
\end{theorem}

\begin{proof} Assume that $\alpha \in \Omega^q _b\left( \mathcal{F}\right) $,
$\alpha =f\cdot \mathrm{vol}^Q$, where $\mathrm{vol}^Q$ is the (globally defined) transverse
volume form. As $\alpha $ and $\mathrm{vol}^Q$ are basic (leafwise invariant)
differential forms, then $f$ will be a basic (\emph{i.e.}  constant along leaves) smooth
function. Locally we may write:
\begin{equation}
\alpha =f\mbox{\,}e^1\wedge \cdots\wedge e^q,  \label{alpha_local}
\end{equation}
with respect to the locally defined basic orthonormal coframe $\left\{ e^i\right\} _{1\le i\le q}$. Clearly
\[
\tilde d_{b,\theta +\frac 12\kappa }\left( \alpha \right) =0,
\]
and let us assume that also
\[
\tilde \delta _{b,\theta +\frac 12\kappa }\left( \alpha \right) =0.
\]

Then we have
\begin{eqnarray*}
\tilde \delta _{b,\theta +\frac 12\kappa }\left( \alpha \right)
&=&-\sum_ie^i\wedge \tilde \nabla _{e_i}\alpha -\iota _{\theta +\frac
12\kappa }\alpha \\
&=&-\sum_ie^i\wedge \nabla _{e_i}\alpha -\iota _\theta \alpha .
\end{eqnarray*}
We consider the local descriptions $\theta =\sum_i\theta _ie^i$ and also (\ref{alpha_local}).
\begin{equation*}
\tilde \delta _{b,\theta +\frac 12\kappa }\left( \alpha \right)
=-\sum_i\iota _{e_i}\nabla _{e_i}\left( f\mbox{\,}e^1\wedge \cdots\wedge
e^q\right)
-\iota _{\theta _ie_i}\left( f\mbox{\,}e^1\wedge \cdots\wedge e^q\right) .
\end{equation*}
From here, as $\mathrm{vol}^Q$ is parallel with respect to $\nabla $, we find
\[
e_i\left( f\right) +f\theta _i=0,
\]
for $1\le i\le q$. As $f$ is basic, we obtain
\begin{equation*}
d_bf+f\theta =\sum_ie^i\wedge \nabla _{e_i}f+\sum_if\theta _ie_i=0.
\end{equation*}
Assuming that the function $f$ is nowhere vanishing, we can write
\begin{equation}
\theta =d\left( -\ln f\right) ,  \label{ln_global}
\end{equation}
which is a contradiction with the initial assumption. Then the zero set of $f $ cannot be empty.

Consider a finite open cover $\left\{ U_i\right\} _{i\in I}$ of $M$,
such that $\left( U_i,\varphi _i\right) $ are contractible foliated local
maps. Then we can find some positive basic functions $\psi _i$ satisfying the
property
\begin{equation}
\theta _{\mid U_i}=-d\left( \ln \psi _i\right) .  \label{ln_local}
\end{equation}
In fact we can construct such function in a canonical way on a local
transversal $T$ such that the above relation is fulfilled for the local
projection of the basic (projectable) one-form $\theta $ on $T$; then we can
take the pull-back of the function on $U_i$, obtaining a local basic
function. Then (\ref{ln_global}) and (\ref{ln_local}) imply $f=\mathbf{c}_i\psi _i$ on $U_i$,
with $\mathbf{c}_i\in \mathbb{R}$. If $f=0$ at some point then necessarily
$\mathbf{c}_i=0$, and $f$ vanishes on a whole open
neighborhood. So the zero set is open. Clearly the zero set of  $f$ is also closed,
and hence it coincides with  the connected manifold $M$, and $\alpha =0$. Consequently there
is no basic harmonic form of degree $q$ with respect to $\tilde \Delta
_{b,\theta }$. Now the Hodge decomposition theorem yields
$\tilde H_{b,\theta }^i\left( \mathcal{F}\right)
=H_{b,\frac 12\kappa +\theta}^q\left( \mathcal{F}\right) =0$.
\end{proof}

\begin{remark}
We point out that a theory of geometric objects which would be
at the same time leafwise invariant and compactly supported is not easy to
undertake, so the extension of the above result to the noncompact case is
not trivial.
\end{remark}

\begin{example} We now present an application for the above theorem. A classical Riemannian
flow can be constructed starting with a matrix $A\in \mathrm{SL}(2,\mathbb{Z})$,
with $\mathrm{Tr}A>2$ \cite{Car,Mo,To}. If $\{ \lambda_i \}_{1\le i \le 2}$
are the eigenvalues of $A$, it is easy to see that
\[
\lambda _i\neq 1,\,\lambda _i>0.
\]
Let $\{ v_i \}_{1\le i \le 2}$  be the corresponding orthonormal
eigenvectors. We denote by $H$ the space $\mathbb{R}^3$ regarded as
\begin{equation}\label{H_identif}
\mathbb{R}^3 \equiv \mathbb{R}\times \mathbb{R}^2
\equiv \mathbb{R}\times \mathbb{R}v_1\times \mathbb{R}v_2
\end{equation}
We endow $H$ with a Lie group structure using the multiplication
\[
p\cdot p^{\prime }:=(t+t^{\prime },\lambda _1^t\alpha ^{\prime }+\alpha
,\lambda _2^t\beta ^{\prime }+\beta )\;,
\]
for any $p=(t,\alpha ,\beta )$, $p^{\prime }
=(t^{\prime },\alpha ^{\prime},\beta ^{\prime }) \in H$, with respect to the identification (\ref{H_identif}) of $\mathbb{R}^3$. Starting with the orthonormal basis $\left\{
e:=(1,0,0),v_1,v_2\right\} $, we construct three left invariant
vector fields on $H$ such that at any point $p$ we have
\begin{eqnarray*}
e_p &=&(1,0,0), \\
v_{1,p} &=&\lambda _1^t(0,1,0), \\
v_{2,p} &=&\lambda _2^t(0,0,1),
\end{eqnarray*}
which, in turn, generate the warped metric
\[
g=\left(
\begin{array}{ccc}
1 & 0 & 0 \\
0 & \lambda _1^{-2t} & 0 \\
0 & 0 & \lambda _2^{-2t}
\end{array}
\right) .
\]
The manner we choose the matrix $A$ insures  that the standard subgroup
$\Gamma :=\mathbb{Z}\times \mathbb{Z}^2$ of $\mathbb{R}^3$ remains a discrete and
cocompact subgroup of $H$. Consequently, we  obtain the quotient Lie group
$\mathrm{T}_A^3:=\Gamma \backslash H$. It also inherits  the above left
invariant geometric objects, which will be denoted in the same way, for
convenience. The flow $\varphi _2$ generated by $v_2$ induces on our manifold
a foliation structure $\mathcal{F}$, which can be proved to be a $GA-$foliated manifold,
where $GA$ is the the \emph{orientation preserving affine
group} \cite{Car}.

We compute the Lie brackets
\begin{eqnarray*}
\lbrack e,v_2] &=&\ln \lambda _1v_2\;, \\
\lbrack e,v_3] &=&\ln \lambda _2v_3\;, \\
\lbrack v_2,v_3] &=&0\;,
\end{eqnarray*}
and we use  the classical Koszul formula for the metric $g$ on $\mathrm{T}_A^3 $ to obtain
\begin{eqnarray*}
g\left( \nabla _{v_2}v_2,e\right) &=&\ln \lambda _2, \\
g\left( \nabla _{v_2}v_2,v_1\right) &=&0.
\end{eqnarray*}
Then $\kappa ^{\sharp }=\ln \lambda _2e$.  As $e^{\flat }=dt$, we
finally find
\[
\kappa =\ln \lambda _2dt.
\]
It is now easy to see that on the compact quotient manifold
$\mathrm{T}_A^3$ the basic form $dt$ is closed but not exact. Furthermore,
$H_b^1(\mathcal{F})\equiv \mathbb{R}$ \cite{Car}, and the basic differential
one-forms $\theta$ that are closed but not exact are precisely
\[
\theta=\mathbf{c}dt+df, \quad \mbox{with\,\,} \mathbf{c}\in \mathbb{R}\setminus\{0\}\mbox{\,and\,}f\in \Omega^0_b(\mathcal{F}).
\]
Applying Theorem \ref{Thm_3} for the above considered Riemannian flow, we obtain
the vanishing of the top dimension group of the basic Morse-Novikov cohomology
complex $\left( \Omega _b\left( \mathcal{F}\right) ,\tilde d_{b,\theta}\right) $, \emph{i.e.}

\[
\tilde H_{b,\theta}^2\left( \mathcal{F}\right) =0.
\]

This result can be regarded as a generalization of \cite[III, Proposition 2]{Car}.
Indeed, for the particular choice $\theta=-1/2\cdot\ln \lambda _2\, dt$, we obtain the vanishing of the basic
cohomology group $H_b^2(\mathcal{F})$, with the corresponding tautness consequences.
For the particular case $\theta=0$ see also \cite{Hab-Ric}.
\end{example}

In the following, we investigate the vanishing of the basic Morse-Novikov
cohomology under the assumption of certain conditions related to the
curvature-type operators. Note that in the classical case
represented by closed Riemannian manifolds, if the de\thinspace Rham cohomology
vanishes, then all closed differential one-forms are in fact exact. The
Morse-Novikov cohomology complex, being isomorphic to de\thinspace Rham complex (as
we noticed in the Section \ref{2.2}), vanishes, too. If the
curvature operator is non-negative and positive at some point, applying a
well known result of Gallot and Meyer \cite{Gal-Mey} on closed Riemannian manifolds, we
obtain that the Morse-Novikov cohomology complex is trivial in this case.

In the general case of a Riemannian foliation (not necessarily with basic mean curvature)
for any $1$\--form $\theta$ which is basic and exact with respect to the operator $d_b$,
we can consider a basic function $f$ with $d_bf=\theta$; then, similar to the classical case
we obtain an isomorphism between the basic de\thinspace Rham and Morse-Novikov cohomologies
using the mapping $\left[ \alpha \right] \rightarrow \left[ e^{-f}\alpha \right] $.

Now, in order to obtain the vanishing of the basic Morse-Novikov cohomology complex
in our framework, the only needed ingredient is the corresponding version of the result
of Gallot and Meyer. This was achieved in \cite{Heb, Mi-Ru-To}, where the authors used
arguments related to functional analysis and operator theory.

Consequently, we easily obtain the following result:

\begin{proposition}\label{Propo}
If a Riemannian foliation $\left( M,\mathcal{F},g\right) $ has
non-negative, and positive at some point, basic curvature operator, then
any closed basic $1-$form $\theta$ is exact, and consequently
\[
H_{b,\theta }^i\left( \mathcal{F}\right) =0,\mbox{\,\,}0<i<q.
\]
\end{proposition}

\begin{remark}
We notice that, in accordance with Remark \ref{remark_weitz}, for the particular case when $\kappa $ is basic
and parallel, the result from \cite{Heb, Mi-Ru-To} can be derived in the classical fashion.
\end{remark}

\begin{example} As an application of Proposition \ref{Propo}, we  consider the case of a suspension
foliation used by Connes \cite{Connes} (see also \cite[Appendix E]{Mo}).
More precisely, let $S$ be a compact orientable surface $S$ of
genus $2$, with universal cover $\tilde S$, and define   $\widetilde{M}:=\mathrm{SO}(3,\mathbb{R})\times \widetilde{S}$.
Let $h:\pi_1(S)\rightarrow \mathrm{SO}(3,\mathbb{R})$ be a group homomorphism. Define a smooth \emph{diagonal}
action of $\pi _1(S)$ on $\widetilde{M}$ by setting
\[
R_{[\gamma ]}(y,\widehat{x})=(h([\gamma ]^{-1})(y),\widehat{x}\circ [\gamma])
\]
for each $[\gamma ]\in \pi _1(S)$.

The quotient manifold  $M:=\widetilde{M}/R$ is then a
$\mathrm{SO}(3,\mathbb{R})$ - foliation. If, moreover, $h$ is injective, then
the leaves are actually diffeomorphic to $\mathbb{R}^2$.

Endow $\widetilde{M}$ with a direct product Riemannian metric, which
is also invariant with respect to the above action $R$. Then the foliated manifold $M$ inherits a
bundle-like metric. Concerning the transverse part,
if the image of $h$ is represented by canonical mappings produced by taking left
multiplications on $\mathrm{SO}(3,\mathbb{R})$, then a left invariant metric
$g$ can be defined in a standard way on our Lie group. For instance, in the
Lie algebra $\mathfrak{o}(3,\mathbb{R})$ we choose $e_1$, $e_2$ and $e_3$ as
\[
e_1:=\left(
\begin{array}{rrr}
0 & -1 & 0 \\
1 & 0 & 0 \\
0 & 0 & 0
\end{array}
\right) ,\mbox{\,\,}e_2:=\left(
\begin{array}{rrr}
0 & 0 & -1 \\
0 & 0 & 0 \\
1 & 0 & 0
\end{array}
\right) ,\mbox{\,\,}e_3:=\left(
\begin{array}{rrr}
0 & 0 & 0 \\
0 & 0 & -1 \\
0 & 1 & 0
\end{array}
\right) ,
\]
and the metric is constructed such that $\{e_i \}_{1\le i \le 3 }$ becomes an orthonormal basis. As
the group is compact, the metric is in fact bi-invariant, and the following
formula can be used to compute the sectional curvature $k$ (see \emph{e.g.}  \cite{Mil}).
\[
k(e_i,e_j)=\frac 14g\left( [e_i,e_j],[e_i,e_j]\right) ,\mbox{\,\,}1\le
i,j\le 3.
\]
As
\[
\lbrack e_1,e_2]=e_3,\mbox{\,\,}[e_2,e_3]=e_1\mbox{\,\,}[e_3,e_1]=e_2,
\]
we obtain
\[
k(e_1,e_2)=k(e_2,e_3)=k(e_3,e_1)=\frac 14,
\]
$\mathrm{SO}(3,\mathbb{R})$ being in fact a distinguished compact Lie group
admitting metrics with strictly positive curvature \cite{Mil}.

Proposition \ref{Propo} now implies that the Morse-Novikov cohomology groups are trivial
for the Connes foliation.
\end{example}

On the other hand, in \cite[Corollary 6.8]{Hab-Ric}, for the particular case when
$\theta \equiv 0$, the authors obtain a vanishing result for the groups of basic
Morse-Novikov cohomology which holds even in the case when the basic curvature operator is not necessarily
non-negative. Note that in this case the basic de\thinspace Rham complex may not be trivial.
Starting from this remark, we obtain the following vanishing result for a
closed basic one-form $\theta $.

\begin{theorem}\label{Thm_4}
Let  $\left( M,\mathcal{F},g\right) $ be a Riemannian foliation
with $M$ closed and the mean curvature $\kappa $ basic. Assume that $\theta $ is a basic closed one-basic form and
define the bi-linear map
\begin{equation} \label{beta_theta}
\beta_{\theta} :\Omega _b\left( \mathcal{F}\right) \times
\Omega _b\left( \mathcal{F}\right) \rightarrow C^\infty \left( M\right),
\mbox{\,with\,} \beta_{\theta} (\cdot ,\cdot ):=\mathcal{L}_{\theta ^{\sharp }}g
\left( \cdot,\cdot \right) +g\left( \mathcal{R}\cdot ,\cdot \right).
\end{equation}
If $\beta $ is non-negatively defined, and $\beta _x$ is positively defined at some point $x\in
M$, then the basic Morse-Novikov cohomology groups vanish
\[
\tilde H_{b,\theta }^i\left( \mathcal{F}\right)=0,\mbox{\,}  0<i<q.
\]
\end{theorem}
\begin{proof}
We start from (\ref{basic_Weitzenbock formula}) and let
$\alpha \in \Omega ^i(\mathcal{F})$, $0<i<q$. Taking integrals on the closed
manifold $M$, we obtain
\[
\left\langle \tilde \Delta _{b,\theta }\alpha ,\alpha \right\rangle
=\sum_i\left\| \tilde \nabla _{e_i}^\theta \alpha \right\| ^2
-2\int_M g(\sum_ie_i\cdot \iota _{\nabla _{e_i}\theta ^{\sharp }}\alpha ,\alpha)
d\mu _g+\int_M g( \mathcal{R}\alpha ,\alpha ) d\mu _g.
\]
where, for arbitrary $\alpha \in \Omega ^i(\mathcal{F})$, we define
$\left\| \alpha\right\|:= \sqrt{\langle\alpha,\alpha}\rangle$.
Then, in order to apply the Bochner technique, we take a closer look at the
middle term. As the scalar product of forms of different degrees vanishes,  Lemma \ref{inter_prod_Clifford} implies:
\begin{eqnarray*}
g( \sum_ie_i\cdot \iota _{\nabla _{e_i}\theta }\alpha ,\alpha ) \
&=&g( \sum_ie^i\wedge \iota _{\nabla _{e_i}\theta ^{\sharp }}\alpha
,\alpha ) \\
\ &=&g( \mathcal{L}_{\theta ^{\sharp }}\alpha ,\alpha )
-g(\nabla _{\theta^{\sharp}} \alpha ,\alpha ) \\
\ &=&-\frac 12\mathcal{L}_{\theta ^{\sharp }}g( \alpha ,\alpha) .
\end{eqnarray*}
This yields the formula
\[
\left\langle \tilde \Delta _{b,\theta }\alpha ,\alpha \right\rangle
=\sum_i\left\| \tilde \nabla _{e_i}^\theta \alpha \right\| ^2+\int_M\left(
\mathcal{L}_{\theta ^{\sharp }}g\left( \alpha ,\alpha \right) +g\left(
\mathcal{R}\alpha ,\alpha \right) \right) d\mu _g.
\]
Arguing now as in the classical case, and considering the  isomorphism
$\mathcal{H}^i( \tilde \Delta _{b,\theta })
\simeq H_{b,\frac12\kappa +\theta }^i\left( \mathcal{F}\right) $, we obtain the result.
\end{proof}

We outline the particular case of Theorem \ref{Thm_4} obtained for $\theta=-1/2 \cdot \kappa$.
Then, the basic Morse-Novikov cohomology complex is again just the classical
basic de\thinspace Rham complex, and we get vanishing conditions suitable for Riemannian foliations with non-positive transverse curvature which are different from the previous results (see \emph{e.g.}  \cite{Heb,Mi-Ru-To}).

\begin{corollary}
If the bi-linear map $\beta_{-\frac{1}{2} \kappa}$ defined in (\ref{beta_theta})
is non-negatively defined and positively defined at some point $x\in
M$, then the basic cohomology groups are trivial,
\[
H_{b}^i\left( \mathcal{F}\right)=0,\quad\mbox{\,}  0<i<q.
\]
Furthermore, as $H^1_b(\mathcal{F})\equiv 0$, the foliation is taut, \cite{Al}.
\end{corollary}

Finally, another particular case is represented by closed Riemannian manifolds.
Again, this classical framework is obtained for the limiting case when the leaves are points.
The mean curvature vanishes, the basic geometric objects become the classical ones,
and we obtain the Bochner technique adapted for Morse-Novikov cohomology.
\begin{corollary}
Let $\left( M,g\right)$ be a closed Riemannian manifold of
dimension $n$ and let $\theta$ be a closed differential one-form. If the bi-linear map
\[
\beta_{\theta} :\Omega\left( M\right) \times
\Omega\left( M\right) \rightarrow \mathcal{C}^\infty \left( M\right),
\mbox{\,with\,\,} \beta_{\theta} (\cdot ,\cdot ):=\mathcal{L}_{\theta ^{\sharp }}g
\left( \cdot,\cdot \right) +g\left( \mathcal{R}\cdot ,\cdot \right),
\]
is non-negatively defined, and positively defined at some point $x\in M$, then
Morse-Novikov cohomology groups vanish:
\[
H_\theta ^i\left( M \right) =0, \mbox{\,for \,}0<i<n.
\]
\end{corollary}

\section{Applications to l.c.s. and l.c.K. foliations}\label{l.c.s._l.c.K.}

In this final section we apply the results obtained in the rest of the paper to the particular
case represented by l.c.s. and l.c.K. foliations.

A \emph{locally conformally symplectic }manifold\emph{\ }(l.c.s.) is a
differentiable manifold $M$ of dimension $2n$ endowed with a differentiable form
$\omega $ of dimension 2 which is  locally conformal with a symplectic form (\emph{i.e.}  closed and non-degenerate
differentiable form of dimension 2) \cite{Vais3}: $d(e^f_U\cdot \omega_{|_U})=0$.   $\omega $ will be also called
\emph{locally conformally symplectic structure}. Following \cite{Drag-Orn, Vais1},
we make the convention that the case when this procedure can be performed
globally is not viewed as a particular case of l.c.s., but as an opposite
case.

A condition equivalent to the definition is the existence of a  global closed
one-form $\theta $ (called \emph{Lee form}) such that
\begin{equation}
d_\theta \omega :=d\omega -\theta \wedge \omega =0.
\label{adapted_cohomology}
\end{equation}

Furthermore, assume the manifold to be complex and endowed with a metric
 $g$ compatible with the complex structure $J$. Then, if $\omega $
is determined by $J$ and $g$ (\emph{i.e.} $\omega (\cdot ,\cdot ):=g(\cdot ,J\cdot)$),
then the manifold is said to be \emph{locally conformally K\"ahler}, in
this latter case a local K\"ahler metric being obtained by a conformal change
of the initial metric. If the Lee form $\theta $ is parallel with
respect to the Levi-Civita connection, then the manifold $M$ is a
\emph{Vaisman manifold} (previously called \emph{generalized Hopf manifold}
\cite{Drag-Orn, Vais2}).

Regarding (\ref{adapted_cohomology}), we see that these types of
differentiable manifolds have a natural Morse-Novikov cohomological complex
(called also \emph{adapted cohomology}), which encodes many interesting
properties of the underlying manifolds (see \cite{Orn-Verb, Vais2}).

The above defined geometric structures can be extended to the context of
Riemannian foliations, the transverse geometry of the foliations
corresponding to the geometry of the manifolds. We thus obtain
\emph{l.c.s. foliations}, \emph{K\"ahler } and \emph{l.c.K. foliations} and, respectively, \emph{Vaisman
foliations} (see \cite{Bar-Drag, Ida-Pop}). We emphasize the interplay
between the above defined spaces: for instance, classical Vaisman manifolds
are examples of K\"ahler foliations of dimension $2$ \cite{Bar-Drag}.

The simplest examples of l.c.s. and l.c.K. foliations are represented by
a smooth Riemannian submersion $f:M\rightarrow N$, the base manifold $N$
being a l.c.s. (l.c.K., respectively) manifold \cite{Bar-Drag}.
In turn, Proposition \ref{Propo} provides a condition for the non-existence of such
transverse structures.

\begin{corollary}
Let $\left( M,\mathcal{F}\right) $ be a foliated manifold of codimension $2q$,
with $M$ compact. Assume there is a bundle-like metric $g$ on $M$ such
that the basic curvature operator is non-negatively defined and positively
defined at some point. Then the foliation does not admit a transverse
locally symplectic structure (and consequently there is no transverse l.c.K.
structure with respect to any bundle-like metric defined on $\left( M,\mathcal{F}\right) $).
\end{corollary}

The proof is straightforward, applying our previous results.

Now, assume that the basic curvature operator allows the existence of a
transverse l.c.s. structure $\omega $ with the transverse Lee form $\theta$.
Then, Theorem \ref{Thm_4} provides vanishing conditions for the basic adapted
cohomology.

\begin{corollary}
If the bi-linear form $\beta_{\theta} $ defined in (\ref{beta_theta}) is
non-negatively defined and positively defined at some point, then the groups
of the adapted basic cohomology vanish
\[
\tilde H_{b,\theta }^i=0\mbox{\,\,for\,\,}0<i<2q.
\]
\end{corollary}

Note that the top dimension cohomology group cannot be addressed with the
above result. In turn, this can be done using Theorem \ref{Thm_3}.

\begin{corollary}
If the basic Lee form $\theta $ is not exact, then $\tilde H_{b,\theta}^{2q}=0$.
\end{corollary}

\begin{remark}
The above result stands as a generalization of \cite[Theorem 2.9]{Drag-Orn}
for the case when the manifold $M$ is compact.
\end{remark}

In the final part of the paper we deal with Vaisman foliations. We first
observe that the suspension foliation constructed in  Example \ref{ex1}, endowed
with a parallel Lee form is consequently a Vaisman foliation. On the other
hand, Theorem \ref{Thm_2} implies that the groups of the basic adapted cohomology are
trivial.

\begin{corollary}
For any Vaisman foliation:
\[
\tilde H_{b,\theta }^i=0, \mbox{\,\,for\,\,}0<i<2q.
\]
\end{corollary}

\begin{remark}
The above result is an extension of \cite{Orn-Verb},
where the triviality of the adapted cohomology is derived directly from the
structure theorem of compact Vaisman manifold \cite{Orn-Verb1}. For the
general case of Riemannian foliations a similar attempt is not a trivial
extension, as structural aspects of Riemannian foliations should be also
considered \cite{Mo}.
\end{remark}

\section*{Acknowledgement}

The authors thank the referee for very carefully reading a first version of the paper and for his most useful suggestions.

\hfill

\hfill

{\small

\noindent {\sc Liviu Ornea\\
University of Bucharest, Faculty of Mathematics, \\14
Academiei str., 70109 Bucharest, Romania. \emph{and}\\
Institute of Mathematics ``Simion Stoilow" of the Romanian Academy,\\
21, Calea Grivitei Str.
010702-Bucharest, Romania }\\
\tt lornea@fmi.unibuc.ro, \ \ liviu.ornea@imar.ro

\hfill

\noindent{\sc Vladimir Slesar\\
Department of Mathematics, University of Craiova,\\
13 Al.I. Cuza Str., 200585-Craiova, Romania}\\
\tt vlslesar@central.ucv.ro
}

\end{document}